\documentclass{amsart}

\usepackage[utf8]{inputenc}
\usepackage{amsfonts,amssymb}

\usepackage{tikz}
\usetikzlibrary{matrix,arrows}

\usepackage{enumerate}
\usepackage{mathrsfs} 
\usepackage{thmtools}
\usepackage{nameref, hyperref, cleveref}

\declaretheorem[numberwithin=section,refname={Theorem,Theorems},  Refname={Theorem,Theorems},name=Theorem]{theorem}
\declaretheorem[sibling=theorem,style=definition,refname={Re\-mark,Re\-marks}, Refname={Remark,Remarks},name=Remark]{remark}
\declaretheorem[sibling=theorem,refname={Pro\-po\-si\-tion,Pro\-po\-si\-tions},  Refname={Proposition,Propositions},name=Proposition]{proposition}
\declaretheorem[sibling=theorem,refname={Lem\-ma,Lem\-mas},  Refname={Lemma,Lemmas},name=Lemma]{lemma}
\declaretheorem[sibling=theorem,refname={Corollary,Corollary},  Refname={Corollary,Corollaries},name=Corollary]{corollary}
\declaretheorem[sibling=theorem,style=definition,refname={De\-fi\-ni\-tion,Def\-initions}, Refname={Definition,Definitions},name=Definition]{definition}
\declaretheorem[sibling=theorem,style=definition,refname={Notation,Notations}, Refname={Notation,Notations},name=Notation]{notation}

\newcommand{\seq}{\subseteq}

\newcommand{\N}{\ensuremath{\mathbb{N}}}
\newcommand{\Va}{\ensuremath{\mathbb{V}}}
\newcommand{\K}{\ensuremath{\mathbb{K}}}
\mathchardef\mhyphen="2D

\DeclareMathOperator{\F}{\mathscr{F}} 
\newcommand{\Lin}{\ensuremath{\mathsf{Lin}}}
\newcommand{\CIRL}{\ensuremath{\mathsf{CIRL}}}
\newcommand{\kRL}{\ensuremath{k\mhyphen\mathsf{CIRL}}}
\newcommand{\kRLsi}{\ensuremath{\kRL_{\textrm{si}}}}
\newcommand{\kRLfin}{\ensuremath{\kRL_{\textrm{fin}}}}
\newcommand{\kRLsifin}{\ensuremath{(\kRL_{\textrm{si}})_{\textrm{fin}}}}

\newcommand{\FL}{\ensuremath{\mathsf{FL}}}

\newcommand{\FLk }{\ensuremath{\mathsf{FL}^{k}_{ew}}}

\DeclareMathOperator{\dmap}{\tikz[baseline]{\draw[>->] (0,.1) to node {\ensuremath{\mathrm{\scriptstyle{D}}}} (0.6,.1);}}
\newcommand{\Dw}{\ensuremath{D^{\wedge}}}
\newcommand{\Dt}{\ensuremath{D^{\to}}}

\renewcommand{\leq}{\leqslant}
\renewcommand{\geq}{\geqslant}

\DeclareMathOperator{\Sub}{\mathsf{Sub}}

\newcommand{\f}{\varphi}

\begin{document}
\title{Canonical formulas for $k$-potent commutative, integral, residuated lattices}
\author[N. Bezhanishvili]{Nick Bezhanishvili}
\email{N.Bezhanishvili@uu.nl }
\address{ILLC, Universiteit van Amsterdam, Science Park 107, 1098XG Amsterdam, The Netherlands.}
\author[N. Galatos]{Nick Galatos}
\email{ngalatos@du.edu}
\address{Department of Mathematics, University of Denver 2360 S. Gaylord St., Denver, CO 80208, USA}
\author[L. Spada]{Luca Spada}
\email{lspada@unisa.it}
\address{Dipartimento di Matematica, Università degli Studi di Salerno. Via Giovanni Paolo II 132, 84084 Fisciano (SA), Italy.}

\thanks{The second-named author acknowledges the support of the grants Simons Foundation 245805 and FWF project START Y544-N23. The third-named author gratefully acknowledges partial support by the Marie Curie Intra-European Fellowship for the project ``ADAMS" (PIEF-GA-2011-299071) and from the Italian National Research Project (PRIN2010--11) entitled \emph{Metodi logici per il trattamento dell'informazione}.}



\begin{abstract}
Canonical formulas are a powerful tool for studying intuitionistic and modal logics. Indeed, they provide a uniform and semantic way of  axiomatising all extensions of intuitionistic logic and all modal logics above K4. Although  the method originally hinged on the relational semantics of those logics, recently it has been completely recast in algebraic terms. In this new perspective canonical formulas are built from a finite subdirectly irreducible algebra by describing completely the behaviour of some operations and only partially the behaviour of some others. 
In this paper we export the machinery of canonical formulas to \emph{substructural logics} by
introducing canonical formulas for $k$-potent, commutative, integral, residuated lattices ($\kRL$).  We show that any subvariety of $\kRL$ is axiomatised by canonical formulas. The paper ends with some applications and examples.
\end{abstract}

\maketitle

\section{Introduction}\label{s:intro}

The apparatus of canonical formulas is a powerful tool for studying intuitionistic and modal logics. We refer to \cite{CZ97} for the details of this method and its various applications.
This technique  relied crucially on the relational semantics of these logics, but recently an algebraic approach to canonical formulas was developed for intuitionistic and modal logics \cite{BB09, BB12, BB11, BBI14}. In this new perspective, the key step is identifying locally finite reducts of modal and Heyting algebras. 

Recall that a variety  $\mathsf{V}$ of algebras is called \emph{locally finite} if its
finitely generated algebras are finite. Although Heyting algebras are not locally finite, their $\vee$-free  and their $\to$-free reducts are locally finite. 
Based on the above observation, for a finite subdirectly irreducible Heyting algebra $A$, \cite{BB09} defined a formula that encodes fully the structure of the $\vee$-free reduct of $A$, and only partially the behaviour of $\vee$. 
Such formulas are called $(\wedge,\to)$-canonical formulas and all intermediate logics can be axiomatised by collections of them. 
In \cite{BB09}, it was shown, via Esakia duality for Heyting algebras, that $(\wedge,\to)$-canonical formulas are equivalent to Zakharyaschev's canonical formulas. 

Recently, \cite{Bezhanishvili2014a} developed a theory of canonical formulas for intermediate logics based on $\to$-free reducts of Heyting 
algebras. 
For a finite subdirectly irreducible Heyting algebra $A$, \cite{Bezhanishvili2014a} defined the $(\wedge,\vee)$-canonical formula of $A$ that encodes fully the structure of the $\to$-free reduct of $A$, and only partially the behavior of $\to$. 
One of the main results of \cite{Bezhanishvili2014a} is that each intermediate logic is axiomatisable by $(\wedge,\vee)$-canonical formulas.

The study of $(\wedge, \to)$-canonical formulas and $(\wedge,\vee)$-canonical formulas leads to new classes of logics with ``good'' properties.
In particular, $(\wedge, \to)$-canonical formulas give rise to \emph{subframe formulas} and $(\wedge,\vee)$-canonical formulas to 
\emph{stable formulas}. These are the formulas that encode only the  $(\wedge, \to)$ and $(\wedge,\vee)$-structures of $A$, respectively. \emph{Subframe logics} and \emph{stable logics} are intermediate logics axiomatisable by subframe and stable formulas, respectively. 
There is a continuum of subframe and stable logics and all these logics have the finite model property \cite[Ch.~11]{CZ97} and \cite{Bezhanishvili2014a}. Stable modal logics also enjoy the 
bounded proof property \cite{BG14}. 

The algebraic approach to canonical formulas opens the way to exporting this method to other non-classical logics where relational semantics have not yet been developed. In this paper we take the first steps in this direction by introducing canonical formulas for a $k$-potent and commutative extension of the \emph{Full Lambek calculus} $\FL$. A proof theoretic presentation of the basic substructural logic $\FL$ is obtained from Gentzen's sequent calculus for intuitionistic logic by removing all structural rules (exchange, weakening and contraction). A \emph{substructural logic} is then any axiomatic extension of the system $\FL$.  
 The logic $\FLk$ under investigation in this paper is an extension of $\FL$ that satisfies exchange, weakening plus the \emph{$k$-potency axiom}: 
\[\underbrace{\f\cdot\ldots\cdot\f}_{k+1\text{ times}}\leftrightarrow \underbrace{\f\cdot\ldots\cdot\f}_{k\text{ times}}.\]

Relational semantics have been developed for $\FL$ in \cite{GJ13} and have been used effectively to establish a series of results, usually relying on insights from proof theory; see for example \cite{APT:LICS, APT:AU, APT:APAL}. However, due to the lack of distributivity they are not amenable directly and easily to some of the methods used in Kripke semantics for intuitionistic logic; for example, due to the lack of distributivity, relational semantics for $\FL$ need to be two-sorted, namely have two sets of possible worlds. In other words, no standard Kripke-style semantics exists for $\FLk$, thus making the algebraic methods used here a natural tool for our study.

The algebraic semantics of substructural logics, known as (pointed) \emph{residuated lattices} were introduced in the setting of algebra well before the connection to logic was established. Residuated lattices appeared first as lattices of ideals of rings, while other examples include lattice-ordered groups and the lattice of all relations on a set. In view of their connection to substructural logics, certain varieties of residuated lattices constitute algebraic semantics for logics such as relevance logic, linear logic, many-valued logics, Hajek's basic logic and intuitionistic logic (in the form of Heyting algebras) to mention a few. They are also related to mathematical linguistics, to $C^*$-algebras and to theoretical computer science. See \cite{GJKO} for more on residuated lattices and substructral logics.

In this paper we introduce $(\vee, \cdot, 1)$-canonical formulas for commutative, integral, $k$-potent residuated lattices ($\kRL$); see \cref{d:canform}.  
  These formulas encode the $(\vee, \cdot, 1)$-structure of a subdirectly irreducible $\kRL$-algebra fully and the structure of $\to$ and $\wedge$ only partially. 
 The key property that makes our machinery work is that the $(\vee, \cdot, 1)$-reducts of $\kRL$-algebras are locally finite \cite{blok2002finite}.
In \cref{t:equiv-n-potent} we show that every extension of $\FLk $ is axiomatisable by such formulas; the main tool towards this result is \cref{thm:main}, that associates to any formula in the language of residuated lattices, an equivalent (finite) set of $(\vee, \cdot, 1)$-canonical formulas.
The two remaining sections are devoted to applications.  In \cref{sec:stable} we study logics whose corresponding classes of subdirectly irreducible algebras are closed under $(\vee, \cdot, 1)$-subalgebras. We call such logics \emph{stable}, and in \cref{cor: stablefmp} we show that all of them have the finite model property and are axiomatised by special $(\vee, \cdot, 1)$-canonical formulas.
In \cref{sec:examples} we give alternative axiomatisations via, $(\vee, \cdot, 1)$-canonical formulas, of some well-known logics extending $\FLk$.

Recently, a classification of formulas in the language of $\FL$, called \emph{substructural hierarchy}, has been introduced \cite{APT:LICS, APT:APAL}.  The classes of this hierarchy are usually denoted by $\mathcal{P}_{i}$ and $\mathcal{N}_{i}$, with $i$ natural number. Their  structure is similar to the time-honoured arithmetical hierarchy.  Axiomatic extensions of $\FL$ by formulas within the first three levels of the hierarchy ($i=0,1,2$) were proved to be particularly amenable  \cite{APT:LICS, APT:APAL}. There has been partial success in the study of the fourth level of the hierarchy, but no progress has been made on the fifth level and up. It follows from the results in this paper that every extension of $\FLk$ can be axiomatised by formulas within the class $\mathcal{N}_{3}$ in the hierarchy, thus providing hope for their thorough understanding.  We remark that after completing this article we learned about a result of Je\v{r}\'{a}bek on the substructural hierarchy which is more general than ours.  Indeed in \cite{Jerabek:15}  Je\v{r}\'{a}bek shows that using a standard technique in proof complexity called \emph{extension variable}, all extensions of $\FL_{e}$ can be axiomatised using formulas up to the level $\mathcal{N}_{3}$.  It should be noted, however, that very little can be said about  the shape of the formulas obtained  in \cite{Jerabek:15}, while we will see in \cref{d:canform} that all canonical formulas share the same uniform shape.  In addition, as shown in \cref{l:A-refutes-gamma}, canonical formulas have a useful semantic characterisation, which will be often exploited in this paper e.g., for establishing the finite model property for large classes of logics.

\section{Preliminaries}\label{s:prelim}

In this section we recall the definition of (commutative) residuated lattices together with some of their basic properties needed in the remainder of the paper.  We start by fixing some notation for standard concepts in universal algebra.

\begin{notation}[Free algebras and valuations]\label{d:free}
Given a variety $\Va$ of algebras  we denote by $\F_{\Va}(\kappa)$ the \emph{free algebra with $\kappa$ free generators} in $\Va$.   When $\Va$ is clear from the context we will omit it and simply write $\F(\kappa)$.
Given an algebra $A$ in a variety $\Va$, a $\Va$-\emph{valuation} (henceforth simply \emph{valuation}) into $A$ is any $\Va$-homomorphism from the algebra of all terms in the language of $\Va$ into the algebra $A$. Since every such a morphism factors through $\F(\omega)$, up to equivalence, we will also think of valuations into $A$ as $\Va$-homomorphisms  form $\F(\omega)$ into $A$. 
Therefore, there is a one-to-one correspondence between the valuations into $A$ and assignments sending the free generators of $\F(\omega)$ into elements of $A$.  We also identify free $n$-generated algebras and the Lindenbaum-Tarski algebras of provably equivalent classes of formulas in $n$ variables, whence we will use the propositional variables $X_{1},\dots,X_{n}$ to indicate the free generators of $\F(n)$.

\end{notation}
In this article we shall mostly need to consider a finite number $n$ of variables, so by an abuse of notation, we shall also call a valuation into $A$ any $\Va$-homomorphism from $\F(n)$ into $A$.  It remains tacitly understood that any extension of such a homomorphism to the algebra $\F(\omega)$ would suit our needs.

We now turn to a brief description of the algebraic semantics of substructural logics.  Recall that $\FL$ is obtained from Gentzen calculus $\mathsf{LJ}$ by dropping the structural rules. Notice that in this way, one ends up with non-equivalent ways of introducing connectives in the calculus.  This entails that a suitable language for $\FL$ is given by two conjunctions $\cdot$ and $\wedge$, a disjunction $\vee$ (\emph{strong} conjunction $\cdot$ distributes over $\vee$, but the \emph{lattice} conjunction $\wedge$ does not),  two implications $/,\backslash$, and  two constants $\underline{0},\underline{1}$; extensions with two additional constants $\underline{\top}, \underline{\bot}$ for the bounds are also considered. 

The equivalent algebraic semantics (in the sense of \cite{blok1989algebraizable}) of $\FL$ is given by the class of (pointed) residuated lattices (see \cite{GJKO} for more details). 
The associated translations between formulas of the logic and equations of the variety is given by the maps 
\begin{align*}
\phi \qquad\qquad&\longmapsto &(1 \leq \phi), \\
(s = t) \qquad\qquad&\longmapsto& (s \leftrightarrow t).
\end{align*} 
Here we identify logical connectives by the corresponding operation symbols in algebra, and logical propositional formulas with algebraic terms. 
Further, since we will be concerned only with extensions that include exchange, only one implication is needed, so we give our definitions in this simpler setting.

\begin{definition}[Residuated lattices]
A commutative residuated lattice is an algebra $\langle A,\cdot,$ $\rightarrow, \wedge,\vee, 1\rangle$ such that
\begin{enumerate}
\item $\langle A,\cdot, 1\rangle$ is a commutative \emph{monoid} i.e., $\cdot$ is commutative, associative and has $1$ as neutral element.
\item $\langle A, \wedge,\vee \rangle$ is a lattice.
\item $\rightarrow$ is the residual of $\cdot$, i.e., 
\[x\cdot y\leq z \quad\text{ iff }\quad y\leq x\rightarrow z,\]
where $x\leq y$ is an abbreviation for $x\wedge y=x$.
\end{enumerate}
A residuated lattice is called: 
\begin{enumerate}[a)]
\item \emph{bounded} if in the order $\leq$ there exist a largest and a least element, denoted by $\top$ and $\bot$, respectively,
\item \emph{integral} if it is bounded and $\top=1$.
\end{enumerate}
A \emph{pointed} residuated lattice is an expansion of a residuated lattice with an additional constant $0$. The constant can be evaluated in an arbitrary way and it is used to define the operation of negation. 
\end{definition}
Notice that despite the presentation given above, residuated lattices form a variety (see \cite[Theorem 2.7]{GJKO} for an equational axiomatisation). One can also easily see that multiplication preserves the order and that it actually distributes over join. If $a$ is an element of a residuated lattice, we write $a^{k}$ for the $k$-fold product $a\cdot\ldots\cdot a$ and $a \leftrightarrow b$ for $(a\rightarrow b) \wedge (b \rightarrow a)$.

Since residuated lattices form the algebraic semantics of $\FL$, an immediate application of \cite[Theorem 4.7]{blok1989algebraizable} tells us that all substructural logics are algebraizable and their equivalent algebraic semantics correspond to subvarieties of (pointed) residuated lattices.  In particular, if $\mathsf{L}$ and $\mathbb{V}_{\mathsf{L}}$ are a logic and a variety that correspond in this way, we have that, for any propositional formula/term $\phi$, $\mathsf{L} \vdash \phi$ iff $\mathbb{V}_{\mathsf{L}} \models \phi\geq1$. Here, as usual, the former means that $\phi$ is a theorem of the logic $\mathsf{L}$, while the latter means that $A \models \phi\wedge 1 =1$ for each $A \in \mathbb{V}_{\mathsf{L}}$.

Let $A$ be a residuated lattice, $v$ a valuation on $A$, and $\f(X_{1},\dots,X_{n})$  a formula in the language of $\FL$. We will write $A\not\models v(\f(X_{1},\dots,X_{n}))$ if $A\models v(\f) < 1_{A}$.  We will write $A\not\models\f(X_{1},\dots,X_{n})$ if there exists a valuation $v$ such that $A\not\models v(\f(X_{1},\dots,X_{n}))$.
In this article we will be mainly concerned with the calculus $\FLk$ which is given by $\FL$ plus exchange $\f\cdot \psi\leftrightarrow \psi\cdot \f$, weakening $\phi \leftrightarrow (\phi \wedge 1)$ and \emph{$k$-potency} $\phi^k \leftrightarrow \phi^{k+1}$.
The equivalent algebraic semantics for $\FLk$ is provided by (pointed) \emph{commutative ($x \cdot y= y \cdot x$), integral ($x \leq 1$), $k$-potent}($x^{k}=x^{k+1}$) residuated lattices; such a class of structure will be denoted by $\kRL$.  Since the results of this paper work independently of the inclusion or not of the constant $0$ in the type we will be informal and drop the adjective `pointed' when we refer to the algebraic semantics for $\FLk$, relying on the reader to fix the correct type on the algebraic or the logical side (so one may  consider either pointed residuated lattices or $\FLk\leq$ without $0$).
\begin{notation}\label{not:RL}
We will denote by {\kRL} the varieties of $k$-potent, commutative, integral, residuated lattices for $k$ ranging among natural numbers.  We will write {\kRLsi} for the class of subdirectly irreducible algebras in {\kRL} and {\kRLfin} for the class of finite algebras in {\kRL}.
\end{notation}

The last part of this section is devoted to recall some result regarding subdirectly irreducible residuated lattices that will be useful in the reminder of the paper.
Recall that an algebra $A$ is called subdirectly irreducible if, whenever it embeds into a direct product of algebras, in such a way that the compositions of the canonical projections with the embedding are still surjective, then $A$ must be isomorphic to one of the algebras in the product (see \cite[Section II.8]{MR648287} for further information).

\begin{definition}[Subcover, coatom, and completely join irreducible] 
Let $\langle A, \leq \rangle$ be any partially ordered set.  If $a,b\in A$, $a< b$ and there is no $c\in A$ such that $a< c<b$ we say that $a$ is a \emph{subcover} of $b$.  If $A$ has a top element, then any subcover of it is called  \emph{coatom}.  
Finally, an element $a\in A$ is said to be \emph{completely join irreducible} if, whenever $a=\bigvee_{i\in I}a_{i}$ with $a_{i}\in A$ there exists $i\in I$ such that $a=a_{i}$.
\end{definition}

\begin{lemma}\mbox{\cite[Lemma 3.59]{GJKO}}\label{lem1}
A finite commutative, integral, residuated lattice is subdirectly irreducible if, and only if, if $1$ is completely join irreducible.
\end{lemma}

The crucial reason for which we restrict to $k$-potent structures is that the above characterisation extends to infinite algebras.  This is observed without a proof in the paragraph subsequent to  \cite[Lemma 3.60]{GJKO}, so we spell out the details here for the sake of completeness.  Before turning to the proof we observe that in the infinite case, having a unique coatom does not imply 1 to be completely join irreducible, for there still can be an infinite chain of elements whose supremum is 1 without any of them being equal to 1.
\begin{remark}\label{r:join-irreducible}
Notice that $1$ is completely join irreducible in an integral commutative residuated lattice $A$ if, an only if, $A$ has a second-greatest element, namely if there is an $a \in A$ such that $\{x \in A : x \not = 1\}=\{x \in A : x \leq a\}$.  For the non-trivial direction, suppose that 1 is completely join irreducible.  If a coatom exists, then it must be unique, for if $a,b$ are two distinct coatoms then $a\vee b=1$, while $a,b\neq 1$, against the completely join irreducibility of 1. If there are no coatoms, then for any $a_{i}\in A$ with $a_{i}\neq 1$ there exists $a_{i+1}$ such that $a_{i}<a_{i+1}<1$.  This would give a sequence of elements, all different from 1, whose supremum is 1, again contradicting the completely join irreducibility of 1.  
\end{remark}

\begin{theorem}\label{t:s-i-opremum}
An algebra $A\in\kRL$ is subdirectly irreducible if, and only if, $1$ is completely join irreducible, or equivalently, if, and only if, $A$ has a second-greatest element.
\end{theorem}
\begin{proof}
The proof that $1$ is completely join irreducible if, and only if, $A$ has a second-greatest element is the content of \cref{r:join-irreducible}.

For the left-to-right implication, suppose $A\in\kRL$ is subdirectly irreducible.  If $A$ is trivial then the claim follows.  
Otherwise, suppose $A$ has no second-greatest element.  As observed in \cite[p.\ 261]{GJKO} if  $a$ and $b$ are two different coatoms in a subdirectly irreducible commutative, integral, residuated lattice $A$, then by \cite[Lemma 3.58]{GJKO} there exists $z < 1$ and natural numbers $m,n$ such that $a^{m} \leq z$ and $b^{n} \leq z$. Let $t=m+n-1$. By the distributivity of $\cdot$ over $\vee$, $(a \vee b)^{t} =\bigvee_{r+s=t}(a^{r} \cdot b^{s})$, where clearly $r,s$ are natural numbers. Note that we cannot have both $r<m$ and $s<n$, since then $r+s< n+m+1=t$. If $r\geq m$ then $a^{r}\cdot b^{s} \leq a^{r} \leq a^{m} \leq z$, while if $s\geq n$, then $a^{r}\cdot b^{s} \leq b^{s} \leq b^{n} \leq z$. Hence $(a\vee b)^{t} \leq z < 1$, which contradicts the fact that $a$ and $b$ are coatoms.

If there are no coatoms at all, then there exists a strictly ascending chain $D$ of elements different from 1, such that $\bigvee D=1$.  Then $(\bigvee D)^{k}=1$.  By residuation $\cdot$ distributes over arbitrary joins, so 
\[\left(\bigvee D\right)^{k}= \left(\bigvee D\right)\left(\bigvee D\right)^{k-1}= \bigvee_{d\in D}\left\{d\cdot \left(\bigvee D\right)^{k-1}\right\},\]
and iterating this we arrive at
\begin{equation}
\label{eq:1}\left(\bigvee D\right)^{k}=\bigvee\left\{{\pi(1)}\cdot \ldots\cdot {\pi(k)}\mid \pi\in \Pi\right\},
\end{equation}
where $\Pi$ is the set of all functions from $k$ into $D$. By commutativity, we can assume ${\pi(1)}\leq\dots\leq {\pi(k)}$.  
For each $i, j\leq k$ we have $\pi(j) \leq 1$ by integrality and further  $\pi({i})\cdot \pi({j})\leq \pi({i})$ by the order preservation of multiplication.  So, for every fixed $\pi$ each factor in the join in \eqref{eq:1} is smaller than ${\pi(k)}^{k}$. 
By \cite[Lemma 3.58]{GJKO} a commutative and integral residuated lattice $A$ is subdirectly irreducible if, and only if, there is an element $a\in A$, $a \not = 1$ such that, for any $b\in A$, $b \not = 1$ there is a natural number $n$ for which $b^{n}\leq a$, and we can take $n$ to be at least $k$, without loss of generality.  Since $A$ is $k$-potent we can actually take $n$ to be equal to $k$.  So, there is an element $a \not = 1$ such that for each $\pi\in \Pi$, ${\pi(k)}^{k}\leq a$. Whence,
 \[\left(\bigvee D\right)^{k} \leq \bigvee\Big\{{\pi(k)}^{k}\mid \pi\in \Pi\Big\}\leq a<1\]  
which contradicts the initial assumption that $\left(\bigvee D\right)^{k}=1$.
Finally, for the right-to-left direction, recall that if the top element 1 in $A$ is completely join irreducible, then by \cite[Lemma 3.59]{GJKO} $A$ is subdirectly irreducible.
\end{proof}

\begin{definition}
A residuated lattice $A$ is said to be \emph{well-connected} if for all $x,y\in A$, $x\vee y=1$ implies $x=1$ or $y=1$.
\end{definition}

\begin{lemma}\label{l:well-connect-subdirectly}
Let $A\in\kRLfin$. Suppose that $B\in\CIRL$ is well-connected, and $h\colon A\to B$ is an injective map that preserves $1$ and binary join in $A$ i.e., $h(1_A) = 1_B$ and  for each $a,b\in A$ we have $h(a\vee b ) = h(a) \vee h(b)$.  Then $A$ is also well-connected, and if it is non-trivial, it is also subdirectly irreducible.
\end{lemma}
\begin{proof}
Suppose that $a,b\in A$ are such that $a\vee b=1_A$, then $h(a)\vee h(b)=h(a\vee b)=h(1_A)=1_B$, and since $B$ is well-connected either $h(a)=1_B$ or $h(b)=1_B$. Since $h$ is injective either $a$ or $b$ must be equal to $1_A$. Finally, since $A$ is finite, well-connected and non-trivial, $1$ is completely join irreducible. By \cref{t:s-i-opremum}, we conclude that  $A$ is subdirectly irreducible. 
\end{proof}

As in the above lemma, in the rest of the article we will need to consider maps between algebras that do not preserve the full signature.  It is then useful to establish a piece of notation for these maps.  
\begin{definition}
Given a signature including the symbols $*_{1},\dots,*_{n}$ and algebras $A$ and $B$  in this signature, we will indicate the fact that a map $f\colon A\to B$ preserves the operations $*_{1},\dots,*_{n}$ by saying that $f$ is a $(*_{1},\dots,*_{n})$-ho\-mo\-morphism, without any further assumption for the remaining operations.  If $f$ is an embedding, we say that $A$ is a $(*_{1},\dots,*_{n})$-subalgebra of $B$.
\end{definition}

\section{Canonical formulas for $k$-potent, commutative, integral, residuated lattices}\label{s:n-pot-RL}
\subsection{$(\cdot,\vee, 1)$-canonical formulas}\label{ss:dot-vee-can-form}
 In this section we introduce $(\cdot,\vee, 1)$-ca\-no\-nical formulas and show that every extension of {\FLk} is axiomatisable by $(\cdot,\vee, 1)$-canonical formulas. 

We first prove an essentially known result about sum-idempotent multiplication $k$-potent commutative semirings. An i-semiring (from idempotent semiring) is an algebra $\langle A,\cdot, \vee, 1\rangle$ where $\langle A, \cdot, 1 \rangle$ is a monoid, $\langle A, \vee \rangle$ is a semilattice and $\cdot$ distributes over $\vee$. An i-semiring is called commutative if $\cdot$ is commutative and $k$-potent if $\cdot$ is $k$-potent. 
\begin{lemma}\label{lem:M(n)} 
Given a commutative $k$-potent i-semiring $B$ and a finite subset $S$ of $B$, the subalgebra $\langle S \rangle$ generated by $S$ has at most $2^{(k+1)^{|S|}}$-many elements. So, the maximal size $M(n)$ of an $n$-generated subalgebra of a commutative $k$-potent i-semiring is at most $2^{(k+1)^n}$.
\end{lemma}

\begin{proof} We assume that $S=\{s_1, \ldots , s_n\}$.
Since multiplication distributes over join, is commutative, and $k$-potent we have $\langle S \rangle = J (Pr (S))$, where $Pr(S) = \{s_1^{p_1} \dots s_n^{p_n}  \mid  0 \leq p_i \leq k$, for $1 \leq i \leq n\}$, and $J(T)=\{\bigvee T_0 \mid T_0 \subseteq T\}$, for $T$ a finite subset of $B$. It is then clear that $|Pr(S)| \leq  (k+1)^n$ and that $|J(T)|\leq |\mathcal{P}(T)|=2^{|T|}$. Thus, $|\langle S \rangle| = |J (Pr (S))| \leq 2^{|Pr(S)|} \leq 2^{(k+1)^n}$.
\end{proof}

 The next lemma was first observed in \cite[Theorem 4.2]{blok2002finite}, we recast it in a way that is expedient to our needs. 

Given a formula $\f$, we denote by $\Sub(\f)$ the collection of all of its subformulas. Further, for an algebra $A$ and a valuation $v$ into $A$, we denote by $\Sub_{v}(\f)$ the set $v[\Sub(\f)]$ of all images in $A$ of subformulas of $\f$. Note that $|\Sub_v(\f)| \leq |\Sub(\f)|$, since some subformulas may attain the same value. 
Given an algebra $B$ and a subset $S$ of $B$, the relational structure that is obtained by the restriction of the operations (viewed as relations) of $B$ on $S$ is called a partial subalgebra of $B$; as it is fully determined by $S$, we will also call it $S$. So, if $f^B$ is an $n$-ary operation on $B$ then $f^B \cap S^{n+1}$ will be an $(n+1)$-ary relation on the subalgebra $S$. Note that all these relations are single-valued but may not be total relations, namely they are partial operations on $S$.

\begin{lemma}\label{lem:4.4}
Let $\f$ be a formula, $B\in\kRL$ and $v$ a valuation into $B$ such that $B\not \models v(\f)$. The partial subalgebra $\Sub_{v}(\f)$ of $B$ can be extended to a finite algebra $A$ in $\kRL$ such that $A$ is a $(\cdot,\vee, 1)$-subalgebra of $B$ and $A\not\models\f$. 
\end{lemma}
\begin{proof}
Let $A$ be the $(\cdot,\vee, 1)$-subalgebra of $B$ generated by $\Sub_v(\f)$.
By \cref{lem:M(n)} the i-semiring $A$ is finite and the following operations are well defined, since the joins are all finite:
 $a,b\in A$
\[a\to_A b:=\bigvee\{c\in A\mid a\cdot c\leq b\} \text{ and } a\wedge_A b:=\bigvee\{c\in A \mid c\leq a\text{ and }c\leq b\}\ .\]
It is well known and easy to verify that under these operations $A$ is actually a residuated lattice. 
 Furthermore, as $a\to b=\bigvee\{d\in B\mid a\wedge d\le b\}$, it is easy to see that $a\to_A b\le a\to b$ and that $a\to_A b=a\to b$ whenever $a\to b\in \Sub_v(\f)$.  The same holds for $\wedge_{A}$.  This entails that $v(\f)$ attains the same value in $A$ and $B$. As $v(\f)\ne 1_B$, we conclude that $v(\f)\ne 1_A$. Thus, $A$ is in $\kRL$, it is a $(\cdot,\vee, 1)$-subalgebra of $B$ and $A$ refutes $\f$. 
\end{proof}

The above lemma motivates the following definition.
\begin{definition}[$(\Dw ,\Dt )$-embedding]\label{d:D-embedding}
Let $A,B\in\kRL$, and $\Dw ,\Dt$ be two subsets of $A^{2}$. We call \emph{$(\Dw ,\Dt )$-embedding} a map $h\colon A\to B$ which is injective, preserves $\cdot$ and $\vee$, and such that if $(a,b)\in \Dt $ then $h(a\to b) = h(a)\to h(b)$ and if $(a,b)\in \Dw $ then $h(a\wedge b) = h(a)\wedge h(b)$. For such maps we use the notation $h\colon A\dmap B$, where $D=(\Dw ,\Dt )$.
\end{definition}

We have now all the ingredients to introduce $(\cdot,\vee, 1)$-canonical formulas. However, to motivate them, we first state the main theorem of this section and then proceed with the formal definition.
\begin{theorem}\label{thm:main}
For any formula $\f$ such that $\FLk \not\vdash\f$, there exists a finite set of $(\cdot,\vee, 1)$-canonical formulas $\{\gamma(A_{i},D_{i}^{\wedge},D_{i}^{\to})\mid 0\leq i \leq m\}$ such that for any $B\in\kRL$
\begin{align}
\label{eq:star}
B\models \f \qquad \text{ if, and only if, }  \qquad \forall i\leq m\quad B\models \gamma(A_{i},D_{i}^{\wedge},D_{i}^{\to})\ .
\end{align}
\end{theorem}
We will see later in \cref{d:algebras-A}, how to associate such formulas with an arbitrary refutable formula $\f$.

\begin{definition}[$(\cdot,\vee, 1)$-canonical formulas]\label{d:canform}
Let $A$ be a finite algebra in $\kRLsi$ and let $\Dw ,\Dt$ be subsets of $A^2$. For each $a\in A$, we introduce a new variable $X_a$, and set
\begin{align*}
\Gamma  := \quad &(X_{\bot} \leftrightarrow \bot) \wedge (X_{1} \leftrightarrow 1) \wedge\\
& \bigwedge \ \{X_{a\cdot b} \leftrightarrow X_a \cdot X_b\mid  a,b\in A\} \ \wedge\\
& \bigwedge \ \{X_{a\vee b} \leftrightarrow X_a \vee X_b\mid  a,b\in A\} \ \wedge\\
& \bigwedge \ \{X_{a\to b} \leftrightarrow X_a\to X_b\mid (a,b)\in \Dt \}\\
& \bigwedge \ \{X_{a\wedge b} \leftrightarrow X_a\wedge X_b\mid (a,b)\in \Dw \}\\
&\text{and}&\\
\Delta := \quad& \bigvee\{X_a\to X_b\mid a,b\in A \mbox{ with } a\not\le b\}.
\end{align*}
Finally, we define the {\em $(\cdot,\vee, 1)$-canonical formula $\gamma(A,\Dw ,\Dt)$} associated with $A$, $\Dw $, and $\Dt $ as 
\[\gamma(A,\Dw ,\Dt ):=\Gamma^{k}\to\Delta.\]
\end{definition}

 In the following we use $C \not \models_1 \Gamma^k \rightarrow \Delta$ to mean that there is a valuation $\mu$ into an algebra $C$ with second greatest element $s_C$ such that $\mu(\Gamma^{k})=1$ and $\mu(\Delta)\leq s_C$. Note that this implies that $C \not \models \Gamma^k \rightarrow \Delta$, since $\mu(\Gamma^{k} \to \Delta)=\mu(\Gamma^{k})\to\mu(\Delta)\le s_C$, thus $\mu$ refutes $\gamma(A,\Dw ,\Dt )$ on $C$.

\begin{lemma}\label{l:A-refutes-gamma}
Let $A$ and $C$ be algebras in $\kRLsi$ with $A$ finite. 
\begin{enumerate}[1.]
\item\label{l:A-refutes-gamma:1} $A  \not \models_{1} \gamma(A,D^{\wedge},D^{\to})$.
\item\label{l:A-refutes-gamma:2}  $A\dmap C$ iff $C \not \models_1 \gamma(A,D^{\wedge},D^{\to})$, for some subsets $D^{\wedge},D^{\to}$ of $A^2$.
\end{enumerate}
\end{lemma}

\begin{proof}
We denote by $s_A$ and $s_C$ the second greatest elements of $A$ and $C$ respectively. 

Item \ref{l:A-refutes-gamma:1} is readily seen by considering the valuation
\begin{align}
\nu(X_{a}):=a.\label{eq:valuation-nu}
\end{align}
Note indeed that the valuation $\nu$ obviously sends to $1_A$ each conjunct of $\Gamma$, so also $\nu(\Gamma)=1_A$, whence $\nu(\Gamma^{k})=\nu(\Gamma)^{k}=1_A$.  To see that $\nu(\Delta)=s_A$, note that for any $a\in A$ such that  $a\neq 1_A$ the implication $X_{1} \to X_{a}$ appears in the join in $\Delta$, so $\nu(\Delta)\geq \nu(X_{1}\to X_{a})$, but $\nu(X_{1}\to X_{a})=1_A\to a=a$, so $\nu(\Delta)\geq \bigvee\{a\in A\mid a\neq 1_A\}=s_A$.  Suppose now toward a contradiction that $\nu(\Delta)=1_A$. Since $A$ has a second greatest element, one of the implications $X_{a}\to X_{b}$ must attain value $1_A$ under $\nu$. But this is not possible as $\nu(X_{a}) = a$, $\nu(X_{b})= b $ and $a\not\leq b$. So $\nu(\Delta)=s_A$.

For the forward direction in item \ref{l:A-refutes-gamma:2}, assume that $h:A\dmap C$. 
We define a valuation $\mu$ on $C$ as the unique extension of the assignment $\mu(X_a):=h(\nu(X_a))=h(a)$ for each $a\in A$, and prove that
 $\mu(\Gamma^{k})=1_{C}$ and $\mu(\Delta)\le s_C$. We first observe that each conjunct in $\Gamma$ is sent into $1_{C}$ by $\mu$. We just treat a couple of representative cases.
\begin{align*}
\mu(X_{1}\leftrightarrow 1)& = \mu(X_{1})\leftrightarrow1_{C}&\text{because $\mu$ is a valuation}\\
&=h(1_{A})\leftrightarrow 1_{C}&\text{by the definition of $\mu$}\\
&=1_{C}\leftrightarrow 1_{C}&\text{because $h$ is a $(\Dw ,\Dt )$-embedding}\\
&=1_{C}\ .
\end{align*}
If the formula $X_{a\wedge b} \leftrightarrow X_a\wedge X_b$ appears among the conjuncts in $\Gamma$, then $(a,b)\in \Dw$. So, reasoning exactly as above, we have
\begin{align*}
\mu(X_{a\wedge b} \leftrightarrow X_a\wedge X_b)&= \mu(X_{a\wedge b}) \leftrightarrow \mu(X_a)\wedge \mu(X_b)\\
&=h({a\wedge b}) \leftrightarrow h(a)\wedge h(b)\\
&=h({a\wedge b}) \leftrightarrow h(a \wedge b)\\
&=1_{C}\ .
\end{align*} 
Now let $a,b\in A$ with $a\not\le b$. Since $h$ is injective, we have $h(a)\not\le h(b)$. Therefore, $\mu(X_a\to X_b)=\mu(X_a)\to\mu(X_b)=h(a)\to h(b)\ne 1_C$. So $h(a)\to h(b)\le s_C$, and hence $\mu(\Delta)\le s_C$. 

For the converse direction, suppose that there exists some valuation $v$ into $C$ such that $v(\Gamma^{k})=1$ and $v(\Delta)\leq s_{C}$. We define $h\colon A \to C$ by $h(a) := v(X_a)$ for each $a\in A$ and show that $h$ is a $(\Dw ,\Dt )$-embedding.
Let $a,b\in A$. Since $v(\Gamma^{k})=1_C$ and $v(\Gamma^{k})\leq v(X_{a\cdot b}) \leftrightarrow (v(X_a) \cdot v(X_b))$, we obtain that $v(X_{a\cdot b}) \leftrightarrow (v(X_a) \cdot v(X_b))=1_C$. Therefore, 
\[h(a\cdot b)=v(X_{a\cdot b}) = v(X_a) \cdot v(X_b)= h(a)\cdot h(b).\] 
By a similar argument, $h({a\vee b}) = h(a) \vee h(b)$, $h(\bot) = v(\bot)$, $h(1_A) = v(1)$, and for $(a,b)\in D_{i}^{\to}$, then $h({a\to b}) = h(a) \to h(b)$ and for $(a,b) \in D_{i}^{\wedge}$ we have $h({a\wedge b}) = h(a) \wedge h(b)$.  To see that $h$ is injective it suffices to show that $a\not\leq b$ in $A$ implies $h(a)\not\leq h(b)$ in $C$. So, suppose $a\not\leq b$. By \eqref{eq:val1}, $v(\Delta)\neq 1_C$,  therefore $v(X_a)\to v(X_b)<1_C$. So $h(a)\to h(b)<1_C$, which implies $h(a)\not\le h(b)$.
\end{proof}

We now explain how to obtain the algebras $A_{i}$'s in the above claim from a formula $\f$.  

\begin{definition}[The system $\{(A_{i},D_{i}^{\wedge},D_{i}^{\to})\mid 1\leq i\leq m\}$ associated with $\f$]\label{d:algebras-A}
Given any formula $\f$ that is not a consequence of {\FLk} we proceed as follows.  
Let $ p = |\Sub(\f)|$. 
Let $(A_1, v_1),\dots,(A_m, v_m)$  be all the pairs such that each $A_i$ is an algebra in $\kRLsi$ whose cardinality, with the notation of \cref{lem:M(n)}, is less or equal than $M(p)$ and $v_i$ is a valuation into $A_i$ such that $A_{i}, v_i\not\models \f$.  We set 
 \begin{align}
\label{eq:def-D1} 
D_{i}^{\wedge}&:=\{(a,b)\in \big(\Sub_{v_{i}}(\f)\big)^2\mid a\wedge b\in \Sub_{v_{i}}(\f)\},\\ 
\label{eq:def-D2} 
D_{i}^{\to}&:=\{(a,b)\in \big(\Sub_{v_{i}}(\f)\big)^2\mid a\to b\in \Sub_{v_{i}}(\f)\}\ .
 \end{align}
We call $\{(A_{i},D_{i}^{\wedge},D_{i}^{\to})\mid 1\leq i\leq m\}$ the \emph{system associated with $\f$}. 
\end{definition}

To prove \eqref{eq:star} we shall go through a further equivalent condition, so in the rest of this section we prove the following equivalences for $B \in \kRL$:
\begin{align}
\label{eq:star2}B\not\models \f \; \Longleftrightarrow\;  \exists i\leq m\,\exists_{SI} C\; A_{i}\dmap  C\twoheadleftarrow B\;\Longleftrightarrow\; \exists i\leq m\; B\not\models \gamma(A_{i},D_{i}^{\wedge},D_{i}^{\to})\ .
\end{align}

\subsection{Proof of \cref{thm:main}}\label{ss:proof-dot-vee-can-form}

\begin{proposition}[First equivalence in \eqref{eq:star2}]\label{thm:4.3}
Suppose $\FLk \not\vdash\f$ and let the system $(A_1,\Dw_1,\Dt_1),$\dots, $(A_m,\Dw_m,\Dt_m)$ be the one associated with $\f$ as in \cref{d:algebras-A}.  For each $B\in\kRL$, the following are equivalent:
\begin{enumerate}[\textup{(}i\textup{)}]
\item\label{thm:4.3:item1} $B\not\models \f$,
\item\label{thm:4.3:item2} $\exists i\leq m\,\exists_{SI} C\; A_{i}\dmap  C\twoheadleftarrow B$.  In other words, there exists $C$, subdirectly irreducible homomorphic image of $B$, and a $(D_{i}^{\wedge},D_{i}^{\to})$-embedding $h\colon A_i\dmap  C$.
\end{enumerate}
\end{proposition}
\begin{proof}
We prove first that \eqref{thm:4.3:item2} implies \eqref{thm:4.3:item1}. Suppose that $h$  is a $(D_{i}^{\wedge},D_{i}^{\to})$-embedding of $A_{i}$ into $C$, where $C$ is a homomorphic image of $B$.  Recalling the definitions in \eqref{eq:def-D1} and \eqref{eq:def-D2}, if $a\to b\in \Sub_{v_i}(\f)$, then $(a,b)\in D_{i}^{\to}$ and if $a\wedge b\in \Sub_{v_i}(\f)$, then $(a,b)\in D_{i}^{\wedge}$.  This entails that $h$ preserves  globally $\cdot$ and $\vee$, and in addition if $a\to b\in \Sub_{v_i}(\f)$, then $h(a\to b)=h(a)\to h(b)$, and if $a\wedge b\in \Sub_{v_i}(\f)$, then $h(a\wedge b)=h(a)\wedge h(b)$. But $v_i(\f)\neq 1$ in $A_i$, so $(h \circ v_i)(\f)\neq 1$ in $C$. Finally, $\f$ fails also in $B$, as $C$ is a homomorphic image of $B$.

For the implication \eqref{thm:4.3:item1} $\Rightarrow$ \eqref{thm:4.3:item2}, suppose $B\not\models\f$. Then, there exists a subdirectly irreducible image $C$ of $B$ with $C\not\models\f$, namely there is a valuation $v$ into $C$ such that $v(\f)\neq 1_C$. 
Let  $S_{C}$ 
be the $(\cdot,\vee, 1)$-subalgebra of $C$ generated by $\Sub_{v}(\f)$. As shown in \cref{lem:4.4}, the set $S_{C}$ 
can be endowed with the structure of a residuated lattice, which is actually in $\kRL$. Furthermore, as $S_{C}$ is a finite $(\cdot,\vee, 1)$-subalgebra of $C$ and $C$ is subdirectly irreducible,  by \cref{l:well-connect-subdirectly} $S_{C}$ is also subdirectly irreducible.
 Recall that $|\Sub_{v}(\f)| \leq |\Sub(\f)|=p$, so $S_{C}$ is generated by at most $p$-many elements, hence $|S_{C}|\leq M(p)$. Since clearly $S_{C}\not\models \f$ we obtain, by \cref{d:algebras-A}, that there is $i\leq m$ such that $S_{C}=A_i$, 
$\Dw_i= \{(a,b)\in (\Sub_{v}(\f))^2\mid  a\wedge b\in \Sub_{v}(\f)\}$ and $\Dt_i= \{(a,b)\in (\Sub_{v}(\f))^2\mid  a\to b\in \Sub_{v}(\f)\}$. Let $h\colon S_{C}\to C$ be the inclusion map. Then by \cref{lem:4.4}, $h\colon S_{C}\dmap C$ . Thus, there is $i\leq m$ and  $h\colon A_i\dmap C$. 
\end{proof}

Having with this concluded the proof of the first equivalence in \eqref{eq:star2}, we now proceed with the second equivalence.

\begin{proposition}[Second equivalence in \eqref{eq:star2}]\label{thm:4}
For any $A\in \kRLsifin$ let $\Dw ,\Dt \subseteq A^2$. For any $B\in\kRL$, the following are equivalent. 
\begin{enumerate}[\textup{(}i\textup{)}]
\item\label{thm:4:item1} $B\not\models\gamma(A,\Dw ,\Dt )$,
\item\label{thm:4:item2} $\exists i\leq m\,\exists_{SI} C\; A_{i}\dmap  C\twoheadleftarrow B$.  Namely, there is  a $(D_{i}^{\wedge},D_{i}^{\to})$-embedding $h\colon A_i\dmap  C$, where $C$ is a subdirectly irreducible homomorphic image of $B$.
\end{enumerate}  
\end{proposition}
\begin{proof}
We first prove \eqref{thm:4:item2} $\Rightarrow$ \eqref{thm:4:item1}.  Suppose that there is a subdirectly irreducible  homomorphic image $C$ of $B$ and $h\colon A\dmap C$. By \cref{l:A-refutes-gamma},   $\gamma(A,\Dw ,\Dt )$ fails on $C$.  Finally we conclude that $B\not\models\gamma(A,\Dw ,\Dt )$ for $C$ is a homomorphic image of $B$.

Now we prove \eqref{thm:4:item1} $\Rightarrow$ \eqref{thm:4:item2}.
With the notation of \cref{d:canform}, our hypothesis is equivalent to $B\not\models\Gamma^{k}\to \Delta$. So, there exists a valuation $v$ into $B$ such that
\begin{equation}
\label{eq:val1}v(\Gamma^{k})\not\leq v(\Delta)\ .
\end{equation} Let $F$ be the filter generated by  $v(\Gamma^{k})$ in $B$. By \cite[page 261]{GJKO} $F=\{b\in B\mid b\geq (v(\Gamma^{k}))^{n}\text{ for } n\in \N \}$.  Notice that by $k$-potency $v(\Gamma^k)^n=(v(\Gamma)^k)^n=v(\Gamma)^k$, so we deduce that $v(\Delta)\not\in F$, for if $v(\Delta)\in F$,  
then $v(\Delta)\geq v(\Gamma^{k})$ and this contradicts \eqref{eq:val1}.  Let $B'$ be the quotient of $B$ modulo $F$, and $q\colon B\twoheadrightarrow B'$ the associated canonical epimorphism, then $q\circ v$ is a valuation into $B'$ such that $q\circ v(\Gamma^{k})=1$ and $q\circ v(\Delta)\neq 1$.  Finally, in all subdirectly irreducible  epimorphic images of $B'$ the element $q\circ v(\Gamma^{k})$ is mapped into 1, while there must exist one in which the element $q\circ v(\Delta)$ is not mapped into 1.  Let 
 $C$ be this subdirectly irreducible  algebra and let $\nu$ be the composition of $q\circ v$ with the canonical quotient of $B'$ into $C$.  
By \cref{l:A-refutes-gamma},  there is a $(\Dw ,\Dt )$-embedding $h\colon A\to C$, where  $C$ is a subdirectly irreducible  homomorphic image of $B$. 
\end{proof}

Combining \cref{thm:4.3} with \cref{thm:4} yields.

\begin{corollary} \label{cor:5.8}
Suppose that $\FLk \not\vdash\f$, then there exist $(A_1,\Dw_1,\Dt_1)$, \dots, $(A_m,\Dw_m,\Dt_m)$ such that each $A_i\in\kRLsifin$, $\Dw_i,\Dt_{i}\subseteq A_i^2$, and for each $B\in\kRL$, we have:
\[
B\models \f(X_1,\dots,X_n) \mbox{ iff } B\models \bigwedge_{j=1}^m \gamma(A_j,\Dw_j,\Dt_j).
\]
\end{corollary}

\begin{proof}
Suppose $\FLk \not\vdash\f$. Set $(A_1,\Dw_1,\Dt_1),$\dots, $(A_m,\Dw_m,\Dt_m)$  as in \cref{d:algebras-A}, in particular $A_j\in\kRLsifin$ and $\Dw_{j},\Dt_{j}\subseteq A_j^2$.  By \cref{thm:4.3}, for each $B\in\kRL$, the fact that $B\not\models\f$ is equivalent to the existence of $i\leq m$, a subdirectly irreducible  homomorphic image $C$ of $B$, and a $(D_{i}^{\wedge},D_{i}^{\to})$-embedding $h\colon A_j\rightarrowtail C$. By \cref{thm:4}, the latter condition is in turn equivalent to the existence of $i\leq m$ such that $B\not\models\gamma(A_j,\Dw_j,\Dt_j)$. Thus, $B\models\f(X_1,\dots,X_n)$ if, and only if, $B\models \bigwedge_{i=1}^m \gamma(A_j,\Dw_j,\Dt_j)$.
\end{proof}

\begin{theorem}\label{t:equiv-n-potent}
Each extension $\mathsf{L}$ of $\FLk $ is axiomatisable by $(\cdot,\vee, 1)$-ca\-no\-nical formulas. Furthermore, if $\mathsf{L}$ is finitely axiomatisable, then $\mathsf{L}$ is axiomatisable by finitely many $(\cdot,\vee, 1)$-canonical formulas.
\end{theorem}

\begin{proof}
Let $\mathsf{L}$ be an extension of {\FLk}. Then $\mathsf{L}$ is obtained by adding $\{\f_i\mid i\in I\}$ to {\FLk} as new axioms. We can safely assume to be in the non-trivial case for which $\FLk \not\vdash\f_i$ for each $i\in I$. The extension $\mathsf{L}$ is axiomatised by the canonical formulas of the systems associated with the $\f_{i}$'s. Indeed, \cref{cor:5.8} entails that for each algebra $B$ and for each $i\in I$, there exist $(A_{i1},\Dw_{i1},\Dt_{i1}),\dots,(A_{im},\Dw_{im},\Dt_{im})$ such that $B\models\f_i$ if, and only if,  $B\models\bigwedge_{j=1}^{m_i}\gamma(A_{ij},\Dw_{ij},\Dt_{ij})$. 
Since each formula gets associated with a finite set of $(\cdot,\vee, 1)$-canonical formulas, the last statement in the theorem also holds, namely if $\mathsf{L}$ is finitely axiomatisable, then $\mathsf{L}$ is axiomatisable by finitely many $(\cdot,\vee, 1)$-canonical formulas.
\end{proof}

\begin{remark}
Note that the crucial property  used in the proof of Theorem~\ref{thm:main} is the local finiteness of the variety of the $(\wedge,\to)$-reducts of algebras in $\kRL$. This implies that $\kRL$ enjoys the \emph{finite embeddability property}. 
Therefore, this strong version of the finite embeddability property via locally finite reducts entails an axiomatisation via canonical formulas. We leave it as an open problem whether there is any connection between the finite embeddability property of a given variety of residuated lattices 
in its general form and axiomatisations via canonical formulas of the corresponding logics.  

\end{remark}

\section{Stable $k$-potent logics}\label{sec:stable}

Fix a finite subdirectly irreducible algebra $A$ in $\kRL$.  Given a $(\cdot,\vee, 1)$-canonical formula $\gamma(A,\Dw, \Dt)$, there are two obvious extreme cases to consider: when $\Dw=\Dt =A^2$ 
and when $\Dw=\Dt = \emptyset$. 

If $\Dw=\Dt = A^2$, then the $(\cdot, \vee, 1)$-canonical formula $\gamma(A,\Dw, \Dt)$
is the so-called \emph{splitting formula} of $A$.
 The terminology is justified by the fact that, if $\mathsf{V}(A)$ is the variety generated by $A$ and $\mathbb{V}_A$ is the variety axiomatised by $\gamma(A,A^2, A^2)$, then $(\mathsf{V}(A), \mathbb{V}_A)$ forms a \emph{splitting pair} in the subvariety lattice of $\kRL$, namely that every subvariety of $\kRL$ is either above $\mathsf{V}(A)$ or below $\mathbb{V}_A$. Indeed, if $\mathbb{V}$ is a subvariety of $\kRL$ that it is not included in $\mathbb{V}_A$, then it contains some algebra $B$ that is not in $\mathbb{V}_A$, namely $B\not\models\gamma(A,A^2, A^2)$. By \cref{thm:4}, for any $B\in\kRLsi$, we have that $B\not\models\gamma(A,A^2, A^2)$ if, and only if,  $A$ is isomorphic to a subalgebra of a homomorphic image of $B$. So, $A \in \mathsf{V}(B) \subseteq \mathbb{V}$, hence $\mathbb{V}$ contains $\mathsf{V}(A)$. That every finite subdirectly irreducible algebra in $\kRL$ defines a splitting was already observed in \cite{GJKO}, but here we give explicitly the corresponding identity axiomatising $\mathbb{V}_A$, which is only alluded to in  \cite{GJKO}. 
 
 The existence of splitting formulas for these logics also follows 
 from \cite[Theorem 2.3]{Citkin14}, where it is  proved that if a variety admits a \emph{ternary deductive term} then one can write a splitting formula for every subdirectly irreducible finitely presented algebra $A$ in this variety.

Splitting formulas and logics axiomatised by them (so-called join-splittings) in the setting of intermediate and modal logics 
have been thoroughly investigated 
(see, e.g., \cite{CZ97} or \cite[Sec.~5.3]{BB09} for a short account). 
For an analysis of splitting algebras in $\CIRL$ we refer to \cite[Ch.\ 10]{GJKO} and ~\cite{kowalski2000splittings}, where it is proven that the only splitting algebra is the $2$-element Boolean algebra.

\medskip

Now we consider the case $\Dw=\Dt = \emptyset$ and show that such formulas axiomatise a continuum of extensions of $\FLk$  with the finite model property. In doing this we follow \cite[Sec.\ 4]{Bezhanishvili2014a},  where the same results are proven for intermediate logics. 

Congruences in (commutative, integral) residuated lattices are in bijective correspondence with certain subsets called \emph{deductive} filters. Given a congruence $\theta$, the corresponding deductive $F_\theta$ is $[1]_\theta$, the equivalence class of $1$; given a deductive filter $F$ the corresponding congruence $\theta_F$ is given by $a \mathrel{\theta_{F}} b$ iff $a \to b, b \to a \in F$.
We begin with some observations on finitely generated deductive filters of algebras in $\kRL$. 
\begin{lemma}\mbox{\cite[p.\ 261]{GJKO}} \label{l:filter-explicit}
Let $A$ be a residuated lattice and let $B\seq A$.  The deductive filter generated by $B\seq A$, denoted by $F(B)$, is given by 
\[F(B)=\left\{x\in A\mid b_{1}\cdot\ldots\cdot b_{n}\leq x \text{ for }b_{1},\dots,b_{n}\in B\right\}.\]
\end{lemma}

\begin{lemma}\label{l:principal-filters}
Let $A\in\kRL$.  
\begin{enumerate}
\item\label{l:principal-filters-item1} Each finitely generated filter of $A$ is a principal lattice filter.
\item\label{l:principal-filters-item2} If $F$ is a finitely generated deductive filter of $A$ and $\theta_{F}$ the associated congruence, then $a \mathrel{\theta_{F}} b$ if, and only if, $d\cdot a=d\cdot b$, where $d=\min F$ is the minimum element of $F$.
\end{enumerate}
\end{lemma}
\begin{proof}
To prove item \ref{l:principal-filters-item1} suppose $B\seq A$ is finite, say $B=\{b_{1},\dots,b_{n}\}$. Let us set $d= b_{1}^{k}\cdot\ldots\cdot b_{n}^{k}$.  By \cref{l:principal-filters} and the fact that $A$ is commutative, integral, and $k$-potent, $d$ is smaller or equal to any product of elements of $B$ and obviously $d\in F(B)$.  Hence we have that $F(B)=\left\{x\in A\mid d\leq x\right\}$.

For item \ref{l:principal-filters-item2}, suppose that $F=\left\{x\in A\mid d\leq x\right\}$, and note that then $d$ has to be idempotent. We have that $a \mathrel{\theta_{F}} b$ if, and only if $a\to b,b\to a\in F$, iff $d\leq a\to b,b\to a$ iff $d\cdot a\leq b$ and $d\cdot b\leq a$ iff $d\cdot a=d \cdot b$. 
For the last equivalence we used that $d \cdot x \leq y$ iff $d \cdot x \leq d \cdot y$, which we justify now. The backward direction follows from the fact that $d \cdot y \leq y$, since $d \leq 1$; the forward direction follows by multiplying by $d$ to obtain $d \cdot d \cdot x \leq d \cdot y$ and using the idempotency of $d$.
\end{proof}

The above lemma can be used to derive a stronger condition from the configuration $A\rightarrowtail C\twoheadleftarrow B$, as we show in the next lemma.
\begin{lemma}\label{lem:4.3}
Let $A,B,C\in\kRLfin$, $A$ be subdirectly irreducible, $f\colon B\twoheadrightarrow C$ be an epimorphism, and $h\colon A\rightarrowtail C$ be a $(\cdot,\vee, 1)$-embedding. Then there exists a $(\cdot,\vee, 1)$-embedding $g\colon A\rightarrowtail B$ such that $f\circ g=h$.
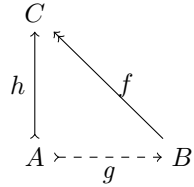
\begin{center}
\begin{tikzpicture}
\matrix(m)[matrix of math nodes, row sep=4em, column sep=4em, ampersand replacement=\&]{
C\& \\
A\& B\\
};
\path[->>] (m-2-2) edge node[right] {$f$} (m-1-1);
\path[>->] (m-2-1) edge node[left] {$h$} (m-1-1);
\path[dashed,>->] (m-2-1) edge node[below] {g} (m-2-2);
\end{tikzpicture}
\end{center}
\end{lemma}
\begin{proof}
Note that, since $B$ is finite, $F=\ker(f)$ is necessarily finitely generated, so it has a minimum element $d$, by \cref{l:principal-filters}(1).
We define $Bd=\{b \cdot d  \mid  b \in B\}$ and we note that it is a $(\cdot, \vee)$-subalgebra of $B$, since $d b_1  \vee d b_2=d(b_1 \vee b_2)$ and $d b_1 \cdot d b_2=d b_1b_2$, by the idempotency of $d$. Recall that $C=B/F=\{[b]_F  \mid  b \in B\}$, and note that by \cref{l:principal-filters}(2) and the idempotency of $d$, for all $b \in B$ we have $b \mathrel{\theta_F} db$, namely $[b]_F=[db]_F$; thus $B/F=\{[db]_F  \mid  b \in B\}=\{[c]_F  \mid  c \in Bd\}$. This proves that the map $\phi:Bd \to B/F$, given by $\phi(x)=[x]_F$ is onto. It is also injective, since $[db_1]_F=[db_2]_F$ implies $db_1 \mathrel{\theta_F} db_2$, which yields $ddb_{1}=ddb_2$, by \cref{l:principal-filters}(2), and $db_1=db_2$, by the idempotency of $d$. By the definition of the operations on $B/F$ it is clear that $\phi$ is then a $(\cdot, \vee)$-homomorphism, so $\phi$ is a $(\cdot, \vee)$-isomorphism. 

Composing the embedding $h:A \to C$, the natural isomorphism $i:C \to B/F$, the $(\cdot, \vee)$-isomorphism $\phi^{-1}:B/F \to Bd$ and the inclusion $in:Bd \to B$, we get a $(\cdot, \vee)$-embedding  $g_d\colon A\to B$; namely $g=in \circ \phi^{-1} \circ i \circ h$. Also, since $[b]_F=[db]_F$, for all $b \in B$, we deduce that  $f \circ (in \circ\phi^{-1} \circ i) = id_C$, and hence $f \circ g_d= f \circ in \circ \phi^{-1} \circ i \circ h= id_C \circ h=h$.

 We now define $g:A \to B$ by $g(1)=1$ and $g(x)=g_d(x)$, otherwise. Note that $1$ is not the result of a product $x \cdot y$ or a join $x \vee y$, for $x, y \in A \setminus \{1\}$, since $A$ is subdirectly irreducible and so well-connected, hence $g$ is still a $(\cdot, \vee)$-homomorphism, but it now also becomes a $(\cdot, \vee, 1)$-homomorphism. Finally, because $1$ is not an element of $Bd$, $g$ is actually a $(\cdot, \vee, 1)$-embedding. Finally, $f\circ g=h$, since $f(g(1_A))=f(1_B)=1_C=h(1_A)$, and for $x \not =1$, $f(g(x))=f(g_d(x))=h(x)$.
\end{proof}

Let $A$ be a finite algebra in $\kRL$. We let $\gamma(A)$ denote the canonical formula $\gamma(A, \emptyset,\emptyset)$.

\begin{theorem}\label{thm:5.2}
Let $A,B\in\kRLsi$, with $A$ finite. Then $B\not\models\gamma(A)$ if, and only if,  there is a $(\cdot,\vee, 1)$-embedding of $A$ into $B$.
\end{theorem}

\begin{proof}
 For the forward direction, if $B\not\models\gamma(A)$, then by \cref{lem:4.4}  there is an $S\in\kRLfin$ which $(\cdot,\vee, 1)$-embeds into $B$ and refutes $\gamma(A)$. Since $B$ is subdirectly irreducible and $S$ is finite, by \cref{l:well-connect-subdirectly}, $S$ is also subdirectly irreducible. Next, by  \cref{thm:4}, there exists a subdirectly irreducible homomorphic image $C$ of $S$ and a $(\cdot, \vee, 1)$-embedding $h\colon A\rightarrowtail  C$. Notice that $C$ is finite, for it is a homomorphic image of $S$.   By \cref{lem:4.3}, $h$ lifts to a $(\cdot, \vee, 1)$-embedding $g\colon A\rightarrowtail S$. Since $S$ is a $(\cdot,\vee, 1)$-sublattice of $B$, we conclude that $g$ is a $(\cdot,\vee, 1)$-embedding of $A$ into $B$. 
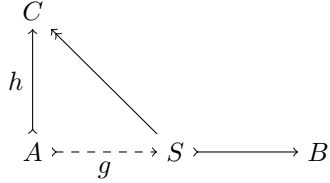
 \begin{center}
\begin{tikzpicture}
\matrix(m)[matrix of math nodes, row sep=4em, column sep=4em, ampersand replacement=\&]{
C\& \&\\
A\& S\& B\\
};
\path[->>] (m-2-2) edge  (m-1-1);
\path[>->] (m-2-1) edge node[left] {$h$} (m-1-1);
\path[dashed,>->] (m-2-1) edge node[below] {$g$} (m-2-2);
\path[>->] (m-2-2) edge  (m-2-3);
\end{tikzpicture}
\end{center}
 The backward direction is an immediate consequence of \cref{l:A-refutes-gamma} with $\Dw=\Dt=\emptyset$.
\end{proof}

\begin{remark}
The reason \cref{thm:5.2} holds only for $\Dw=\Dt=\emptyset$ is that if $\Dw\ne\emptyset$ or $\Dt\ne\emptyset$, then the $(\cdot,\vee, 1)$-embedding $g\colon A\rightarrowtail B$ constructed in the proof of \cref{lem:4.3} may not preserve implications from $\Dt$ or meets from $\Dw$ even if $h\colon A\rightarrowtail C$ preserves them.
\end{remark}

We are ready to introduce stable extensions of $\FLk$.

\begin{definition}\label{def:5.4}
Let $\Va$ be a subvariety of $\kRL$.
We call $\Va$ \emph{stable} if the class $\Va_{si}$ of its subdirectly irreducible algebras is closed under subdirectly irreducible  $(\cdot,\vee, 1)$-subalgebras, namely
if $B \in \Va_{si}$, $A \in\kRLsi$ and $A$ is a $(\cdot,\vee, 1)$-subalgebra of $B$, then $A \in \Va_{si}$.
Equivalently, since $\Va_{si}$ is closed under isomorphisms, the condition can be phrased in terms of $(\cdot,\vee, 1)$-embeddings, namely whenever $A,B\in\kRLsi$ and $h\colon A\rightarrowtail B$ is a  $(\cdot,\vee, 1)$-embedding, $B\in \Va$ entails $A\in \Va$.  Let $\mathsf{L}$ be an extension of  $\FLk $. We say that $\mathsf{L}$ is stable if the equivalent algebraic semantics $\Va_{\mathsf{L}}$ of $\mathsf{L}$ is stable. 
\end{definition}

It can be easily seen that stable extensions include all subvarieties axiomatised by $(\cdot, \vee, 1)$-equations. The latter ones correspond to simple structural rules, when considering extensions of $\FLk $, and it is known, see for example \cite{GJ13}, that they all have the finite model property. Here we extend this result. 

\begin{theorem}\label{cor: stablefmp}
Each stable extension of  $\FLk $  has the finite model property.
\end{theorem}

\begin{proof}
Let $\mathsf{L}$ be a stable extension of  $\FLk $ 
and let $L\not\vdash\varphi$. Then, by Birkhoff's theorem, there exists a subdirectly irreducible $B\in\Va_{\mathsf{L}}$ such that $B\not\models\varphi$. By \cref{lem:4.4}, there exists $A\in\kRL$ such that $A$ is a bounded $(\cdot, \vee,1)$-subalgebra of $B$ and $A\not\models\varphi$. 
Moreover, as $B$ is subdirectly irreducible, by \cref{l:well-connect-subdirectly} so is $A$. Since  $\Va_{\mathsf{L}}$ is stable, $A\in\Va_{\mathsf{L}}$, and as $A$ is finite and $A\not\models\varphi$, we conclude that $\mathsf{L}$ has the finite model property.
\end{proof}

In order to axiomatise stable $k$-potent logics, we recall the theory of frame-based formulas of \cite{Nic06, Nic08}. Although the theory was developed for frames, as was pointed out in \cite{BB12}, dualising frame-based formulas yields algebra-based formulas that we define here in the context of residuated lattices. 

\begin{definition}\label{d:algebra-order}
Let $\K$ be a class of s.i.\ residuated lattices.  We call $\precsim$ an \emph{algebra order} on $\K$ if it is a reflexive and transitive relation on $\K$ and has the following properties:
\begin{enumerate}
\item\label{d:algebra-order:item1} If $A,B\in\K$, $B$ is finite, and $A\prec B$, then $|A|< |B|$, where $A\prec B$ means that $A\precsim B$ and $A$ is not not isomorphic to $B$.
\item\label{d:algebra-order:item2} If $A\in\K$ is finite, then there exists a formula $\zeta(A)$ such that for each $B\in\K$, we have $A\precsim B$ if, and only if,  $B\not\models\zeta(A)$.
\end{enumerate}
The formula $\zeta(A)$ is called the \emph{algebra-based formula} of $A$ for $\precsim$.  
\end{definition}

The following criterion of axiomatisability by algebra-based formulas is a straightforward generalisation of \cite[Theorem~3.9]{Nic08} (see also \cite[Theorem~3.4.12]{Nic06} and \cite[Theorem~7.2]{BB12}). 
\begin{theorem}\label{thm:nick}
Let $\K\subseteq\K'$ be classes of s.i.\ residuated lattices and $\precsim$ be an algebra order on $\K'$. Then $\K$ is axiomatised, relatively to $\K'$, by algebra-based formulas for $\precsim$ if, and only if, 
\begin{enumerate}[\textup{(}a\textup{)}]
\item\label{thm:nick:item1} $\K$ is a down-set of $\K'$ with regard to $\precsim$.
\item\label{thm:nick:item2} For each $B\in\K'\setminus \K$, there exists a finite $A\in\K'\setminus\K$ such that $A\precsim B$.
\end{enumerate}
If \eqref{thm:nick:item1} and \eqref{thm:nick:item2} are satisfied, then $\K$ is axiomatised by the algebra-based formulas of the $\precsim$-minimal elements of $\K'\setminus\K$.
\end{theorem}
\begin{proof}
For the forward direction, suppose $\K$ is axiomatised, relatively to $\K'$, by algebra-based formulas for $\precsim$.  Let $\{\zeta(A_{i})\mid i\in I\}$ be such an axiomatisation for $\K$ with $\{A_{i}\mid i\in I\}$ a family of finite, s.i.\ algebras in $\K'$.  We start by showing that $\K$ is a $\precsim$-down set.  Suppose that $A,B\in \K'$,  $A\precsim B$, $B\in\K$ and, by way of contradiction, $A\not\in \K$.  Then, there exists some $i\in I$ such that $A\not\models \zeta(A_{i})$.  So, by \cref{d:algebra-order} item \ref{d:algebra-order:item2}, $A_{i}\precsim A$ and by transitivity $A_{i}\precsim B$.  Again by \cref{d:algebra-order} item \ref{d:algebra-order:item2}, the latter fact gives $B\not\models\zeta(A_{i})$, which contradicts $B\in \K$.  Thus $\K$ is a $\precsim$-down set as in \eqref{thm:nick:item1}.  Similarly, if $B\in\K'\setminus \K$,  then there exists $i\in I$ such that $B\not\models\zeta(A_{i})$, so by \cref{d:algebra-order} item \ref{d:algebra-order:item2}, $A_{i}\precsim B$.  Notice that $A_{i}$ is finite s.i.\ and does not belong to $\K$, as by reflexivity $A_{i}\precsim A_{i}$, so $A_{i}\not\models\zeta(A_{i})$. This shows that also \eqref{thm:nick:item2} holds.

For the converse direction, suppose that \eqref{thm:nick:item1} and \eqref{thm:nick:item2} hold and consider the axiomatisation
\begin{align}
\label{eq:axioms}\{\zeta(A_{i})\mid A_{i} \text{ is } a \precsim\text{-minimal element of }\K'\setminus\K\}.
\end{align}
We prove that $\K$ is axiomatised by \eqref{eq:axioms}.  Let $A\in\K$ and $A_{i}$  be an arbitrary $\precsim$-minimal element of $\K'\setminus\K$.  Since by \eqref{thm:nick:item1} $\K$ is a down set, $A_{i}\not\precsim A$.  But then, by \cref{d:algebra-order} item \ref{d:algebra-order:item2},  $A\models\zeta(A_{i})$.  As $A_{i}$ was arbitrary, $A$ validates  all formulas in \eqref{eq:axioms}.  Vice versa, if $A\not\in\K$ then by  \eqref{thm:nick:item2}, there exists a finite $B\in\K'\setminus\K$ such that $B\precsim A$.  Suppose that there is $C\precsim B$, then either $C$ is isomorphic to $B$ or $C\precsim B$, hence by \cref{d:algebra-order:item1} in \cref{d:algebra-order}, $|C|<|B|$. Since $B$ is finite, there must be a $\precsim$-minimal algebra below $B$, say $A_{i}$, such that $A_{i}\precsim B$.  
But then, by transitivity, $A_{i}\precsim A$. Therefore, $A\not\models\zeta(A_{i})$, which finishes the proof. 
\end{proof}

We are ready to prove that stable $k$-potent logics are axiomatised by formulas of the form $\gamma(A)$.

\begin{theorem}\label{thm:5.4}
An extension $\mathsf{L}$ of $\FLk$ is stable if, and only if,  there is a family $\{A_i\mid i\in I\}$ of algebras in $\kRLsifin$ such that $\mathsf{L}$ is axiomatised by $\{\gamma(A_{i})\mid i\in I\}$.
\end{theorem}

\begin{proof}
First suppose that there exists a family $\{A_i\mid i\in I\}$ of algebras in $\kRLsifin$ such that $L = \FLk   + \{\gamma(A_{i})\mid i\in I\}$. Let $A,B\in\kRLsi$, $h\colon A\rightarrowtail B$ be a $(\cdot,\vee, 1)$-embedding, and $B\in\Va_{\mathsf{L}}$. If $A\notin\Va_{\mathsf{L}}$, then there exists $i\in I$ such that $A\not\models\gamma(A_i)$. By \cref{thm:5.2}, there exists a $(\cdot,\vee, 1)$-embedding $h_i\colon A_i\rightarrowtail A$. Therefore, $h\circ h_i$ is a 
$(\cdot,\vee, 1)$-embedding of $A_i$ into $B$. Applying \cref{thm:5.2} again yields $B\not\models\gamma(A_i)$, so $B\notin\Va_{\mathsf{L}}$. 
The obtained contradiction proves that $\Va_{\mathsf{L}}$ is stable. We conclude that $\mathsf{L}$ is stable.

Conversely, suppose that $\mathsf{L}$ is stable. Define $\precsim$ on $\kRLsi$ by $A\precsim B$ if there is a $(\cdot,\vee, 1)$-embedding from $A$ into $B$. It is straightforward to see that $\precsim$ is reflexive and transitive. To see that $\precsim$ is an algebra order, observe that condition (1) of \cref{d:algebra-order} is satisfied trivially. For condition (2), if $A,B\in\kRLsi$ with $A$ finite, \cref{thm:5.2} yields that $A\precsim B$ if, and only if,  $B\not\models\gamma(A)$. Therefore, $\precsim$ is an algebra order on $\kRLsi$ and $\gamma(A)$ is the algebra-based formula of $A$ for $\precsim$. It remains to verify that $\precsim$ satisfies conditions (a) and (b) of \cref{thm:nick}. Since $\Va_{\mathsf{L}}$ is stable, $(\Va_{\mathsf{L}})_{{\sf si}}$ is a down-set of $\kRLsi$, and so $\precsim$ satisfies condition (a). For condition (b), let $B\in\kRLsi\setminus(\Va_{\mathsf{L}})_{{\sf si}}$. Then $B\not\models\varphi$ for some theorem $\varphi$ of $\mathsf{L}$. By \cref{lem:4.4}, there is $A\in\kRLsifin$ such that $A$ is a $(\cdot,\vee, 1)$-sublattice 
of $B$ and $A\not\models\varphi$. This implies that $A\in\kRLsi\setminus(\Va_{\mathsf{L}})_{{\sf si}}$ and $A\precsim B$. Thus, $\precsim$ satisfies condition (b), and hence, by \cref{thm:nick}, the family 
\[\{\gamma(A)\mid A\text{ is a $\precsim$-minimal element of }\kRLsi\setminus(\Va_{\mathsf{L}})_{{\sf si}}\}\] axiomatises $\mathsf{L}$.
\end{proof}

We note that using the normal form representation given in  \cite{APT:LICS, APT:APAL} it is easy to see that each
formula appearing on level $\mathcal{P}_3$ of the substructural hierarchy is provably equivalent over intuitionistic logic  to an
$\mathsf{ONNILLI}$-formula \cite{BdJ16}. Consequently, all formulas in the class $\mathcal{P}_3$ axiomatise stable intermediate logics 
\cite[Thm. 5]{BdJ16}. It is an open question whether all stable intermediate logics are axiomatisable by $\mathcal{P}_3$-axioms. We leave open the questions whether 
every $\mathcal{P}_3$-formula gives rise to stable extensions of $\FLk$ and whether all 
stable extensions of $\FLk$ are axiomatisable by $\mathcal{P}_3$-formulas.

We conclude this section by noting that the cardinality of stable extensions of $\FLk$ is that of the continuum. This result directly follows from the fact that there is already 
a continuum of stable extensions of the intuitionistic propositional calculus {\sf IPC}, and {\sf IPC} is an extension of $\FLk$. One may also wonder what is the cardinality of the interval between $\FLk$ and ${\sf IPC}$. For showing that there is a continuum of 
such logics it is sufficient to construct an infinite $\precsim$-antichain of algebras $\kRLsifin$ that are not Heyting algebras, and then apply the standard argument using stable formulas
(see e.g., \cite{Jan68},  \cite[Theorem~11.19]{CZ97}, \cite[Theorem~3.14]{Nic08}, \cite[Theorem~3.4.18]{Nic06}, \cite{Bezhanishvili2014a}). Such anti-chains are easier to construct in the varieties 
of Heyting and modal algebras since these algebras admit a dual representation via finite Kripke frames and for these structures the techniques of  combinatorial 
set-theory apply. 
While conjecturing that such an antichain exists, we leave it as an open problem here.

\section{Examples}\label{sec:examples}
We give here some applications of the results in the previous sections.  
\subsection{Pre-linear $k$-potent commutative, integral, residuated lattices}\label{ss:pl-k-RL}
Consider the class $\Lin$ of linearly ordered algebras in $\kRL$.  We illustrate our results by providing an alternative to the known (see \cite{GJKO}, for example) axiomatisation for the variety $\Va(\Lin)$ generated by $\Lin$.  

It is known that the subdirectly irreducible algebras in $\Va(\Lin)$ are linearly ordered, see \cite{GJKO}, for example.
Consider now the following lattices:
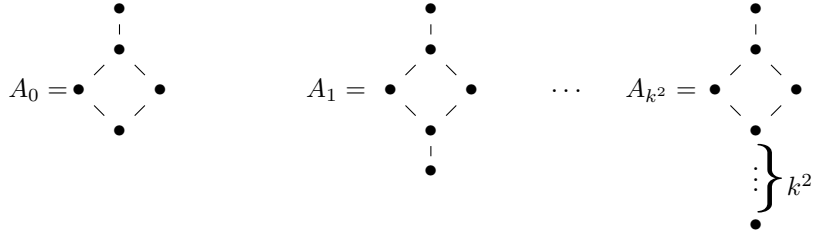
\begin{center}
\begin{tikzpicture}[scale=0.18]
\node (name1) at (-3,3)   {$A_{0}=$};
\node (A11) at (3,0)   {$\bullet$};
\node (A12) at (0,3) {$\bullet$};
\node (A13) at (3,6) {$\bullet$};
\node (A14) at (6,3) {$\bullet$};
\node (A15) at (3,9)   {$\bullet$};
\draw [-] (A11) -- (A12) -- (A13) -- (A14) -- (A11);
\draw [-] (A13) -- (A15);
\node (name2) at (19,3)   {$A_{1}=$};
\node (A20) at (26,-3)   {$\bullet$};
\node (A21) at (26,0)   {$\bullet$};
\node (A22) at (23,3) {$\bullet$};
\node (A23) at (29,3) {$\bullet$};
\node (A24) at (26,6)   {$\bullet$};
\node (A25) at (26,9)   {$\bullet$};
\draw [-] (A21) -- (A22) -- (A24) -- (A23) -- (A21);
\draw [-] (A24) -- (A25);
\draw [-] (A20) -- (A21);
\node (name2) at (36,3)   {$\ldots$};
\node (name2) at (43,3)   {$A_{k^{2}}=$};
\node (dots) at (50,-7)   {$\bullet$};
\node (A30) at (50,-3)   {$\vdots$};
\node (A31) at (50,0)   {$\bullet$};
\node (A32) at (47,3)   {$\bullet$};
\node (A33) at (53,3)   {$\bullet$};
\node (A34) at (50,6)   {$\bullet$};
\node (A35) at (50,9)   {$\bullet$};
\draw [-] (A31) -- (A32) -- (A34) -- (A33) -- (A31);
\draw [-] (A34) -- (A35);
\node (A35) at (52,-3.5)   {$\resizebox{1.3em}{!}{\}}k^{2}$};
\end{tikzpicture}

Let $\mathbb{A}_i$ denote the class of all algebras in $\kRL$ whose lattice reduct is $A_i$.

\end{center}
\begin{lemma}\label{l:k-lemma}
An algebra $B\in\kRLsi$ does not belong to $\Va(\Lin)$ if, and only if, for some $A \in \mathbb{A}_{k^2}$,  $A\dmap B$, where $D=(\emptyset,\emptyset)$.
\end{lemma}
\begin{proof}
 Every $\vee$-embedding is clearly also an order embedding (it preserves and reflects the order), so clearly none of the algebras based on $A_{k^2}$  $D$-embeds in $B$.

Vice versa, if $B\not\in \Va(\Lin)$ then $B$ is not linearly ordered, hence there must be at least two incomparable elements in $B$.  Consider the set $Y$ of all possible product combinations of these two elements.  By $k$-potency the set $Y$ is finite, hence there must exist elements $c,d$ which are minimally incomparable, i.e.\ $c$ and $d$ are incomparable and there exists no pair of incomparable elements $e,d$, with $e <c$ or $d<b$. Notice that if an element $e$ is below either $c$ or $d$, then it must be also below the other, for otherwise the new pair given by $e$ and the incomparable element would contradict the minimality of $c,d$; so the sets $\{b \in B  \mid  b<c\}$ and $\{b \in B  \mid  b<d\}$ are equal and totally-ordered.  We claim that the set
\[J:=\{1,c\vee d\}\cup\{e\in Y\mid e\leq c, d\}\]
is a $(\cdot,\vee, 1)$-subalgebra of $B$.  The closure under $\vee$ and 1 is obvious. To see that it is also closed under $\cdot$ notice that $c^{m}\cdot d^{m}\leq c, d$ for all $m\leq k$ and $c\cdot (c\vee d)= c^{2}\vee cd$, where both $c^{2}$ and $cd$ are below $c$, hence their join belongs to $J$.  It is straightforward that the cardinality of $J$ cannot exceed $k^{2}+2$. So $J$ is isomorphic to one of the $A_{i}$ in our list.

 The result follows from seeing that every algebra in $\mathbb{A}_i$ for $i \leq k^2$ embeds in some algebra in $\mathbb{A}_{k^2}$. Indeed, given any algebra $A$ in $\kRL$ we can construct a new algebra $2[A]$ (also denoted by $2 \oplus A$) that has one new bottom element, is still in $\kRL$ and has $A$ as a subalgebra; see \cite{GJKO} for details. Iterating this construction we see that we can construct an algebra based on $A_{k^2}$ as a superalgebra. 
\end{proof}

\begin{theorem}\label{t:prelinear-axiom}
The variety $\Va(\Lin)$ is axiomatised over $\kRL$ by $\{\gamma(A)\mid A \in \mathbb{A}_{k^2}\}$.
\end{theorem}
\begin{proof}
Call $G$ the variety axiomatised by the above set of formulas.  Notice that, by \cref{thm:5.2}, a subdirectly irreducible algebra $B$ belongs to $G$ if, and only if, for no $A \in \mathbb{A}_{k^2}$ does it happen that $A\dmap B$.  By \cref{l:k-lemma} this happens if, and only if, $B$ is a subdirectly irreducible algebra in $\Va(\Lin)$.  So, the subdirectly irreducible algebras in  $G$ and $\Va(\Lin)$  coincide and this readily implies  that $G=\Va(\Lin)$.
\end{proof}

We obtain directly from \cref{t:prelinear-axiom} and \cref{cor: stablefmp} the following known result; see \cite{GJKO}.

\begin{corollary}
The variety $\Va(\Lin)$ has the finite model property.
\end{corollary}

This example also underlines the more complex behaviour of $\kRL$ compared to Heyting algebras.
Note indeed, that the variety generated by linear Heyting algebras is axiomatised by taking the stable formulas of only $A_{0}$ and $A_{1}$ \cite{Bezhanishvili2014a}. 

\subsection{Pre-linear $\kRL$-algebras of bounded height}\label{ss:bounded-pl-k-RL}

We want to axiomatise the variety generated by the class $\Lin_{\leq h}$ of all linearly ordered algebras in $\kRL$ of cardinality at most $h$. Actually, the variety is also generated by the class $\Lin_{h}$ of all linearly ordered algebras in $\kRL$ of cardinality exactly $h$, given the construction $A \mapsto 2[A]$ mentioned in the last proof, under which every algebra in $\Lin_{\leq h}$ can be embedded in an algebra in $\Lin_h$.   As in the previous subsection, we look for a minimal set of algebras $S$ such that for any $B\in\kRLsi$, $B\in\Va(\Lin_h)$ if, and only if, none of the algebras in $S$ $D$-embeds into $B$, with $D=(\emptyset,\emptyset)$.  It is easy to see that in the case of Heyting algebras it suffices to take as elements of $S$ the algebras $A_{1}$ and $A_{2}$ from the previous subsection, plus the linearly ordered Heyting algebra with $h+1$ elements. In our case, there are numerous linearly ordered algebras in $\kRL$ with $h+1$ elements, forming the class $\Lin_{h+1}$.

\begin{lemma}\label{l:k-lemma2}
An algebra $B\in\kRLsi$ does not belong to $\Va(\Lin_h)$ if, and only if, some algebra in $\Lin_{h+1} \cup \mathbb{A}_{k^2}$ $D$-embeds into $B$, where $D=(\emptyset,\emptyset)$.
\end{lemma}
\begin{proof}
 The subdirectly irreducible algebras $\Va(\Lin_h)$ are totally ordered, so clearly no algebra in $\mathbb{A}_{k^2}$ embeds into any of them. Also, no algebra in  $\Lin_{h+1}$ embeds either, as it has more elements.

 Conversely, if $B\not\in \Va(\Lin_h)$ then either $B$ is not linearly ordered, hence some algebra from $\mathbb{A}_{k^2}$ can be embedded in it, as seen in the proof of \cref{l:k-lemma}, or otherwise $B$ is linearly ordered with more than $h$ elements. Consider the bottom $h$-many elements of $B$ together with $1_B$, and note that they form a $(\cdot, \vee, 1)$ subalgebra of $B$ and they also can be uniquely expanded into an algebra in $\Lin_h$.
\end{proof}

As above we can obtain the following result.

\begin{theorem} 
The variety $\Va(\Lin_h)$ is axiomatised over $\kRL$ by $\{\gamma(A)\mid A \in \Lin_h \cup \mathbb{A}_{k^2}\}$. The variety $\Va(\Lin_h)$ has the finite model property.
\end{theorem}

 The above results are sensitive to the absence of a bottom element in the signature, as it allows us to embed $A_i$ into $A_j$, for $i \leq j$, and similarly for the case of $\Lin_{\leq h}$. In case we have the bottom element in the signature the results need to be modified slightly to consider all the algebras in $\kRL$ that are based on some $A_i$, for $i \leq k^2$. This is actually already noticeable for Heyting algebras, for which both $A_2$ and $A_1$ need to be considered, while for the bottom-free reducts, known as Brouwerian algebras, just $A_2$ would be enough.

\section{Further directions}
We conclude the paper with a list of possible future generalisations and open problems.
\begin{enumerate}
\item One can try to drop integrality $x \leq 1$,  as we can use  \cite[Lemma 3.60]{GJKO} to obtain that $1$ has a unique second-last element $s$. Now, $1 \not \leq x$ does not imply $x \leq s$, however, it means that $x \wedge 1 \leq s$, so one can modify the canonical formulas accordingly by adding a $\wedge 1$.
 \item  The $(\wedge, \to, \bot)$-fragment of Heyting algebras has been used also to find different canonical formulas \cite{BB09}.  We wonder what would be the equivalent of that in the case of $\kRL$.
\item  Dropping commutativity, one can still get local finiteness from $n$-potency and e.g., the following axiom: $xyx=xxy$.  However we do not know whether subdirectly irreducible algebras in this class can still be characterised as the ones with a unique second-last element.
\item In order to remove the need for a unique second-last element one can work with canonical rules instead of canonical formulas; see \cite{Jer09} and 
\cite{BBI14} for similar results for modal logics. In \cite{Jer09} and \cite{BGGJ15} these canonical rules are used 
for obtaining bases for admissible rules for transitive modal logics and intermediate logics. Furthermore, these rules yield alternative proofs of the decidability of the admissibility problem for these logics. Therefore, once canonical rules for $\kRL$ and related substructural logics are defined, the natural next step is to investigate 
 whether these rules could be used to study admissible rules for these logics. 
\end{enumerate}


\begin{thebibliography}{10}
\providecommand{\url}[1]{{#1}}
\providecommand{\urlprefix}{URL }
\expandafter\ifx\csname urlstyle\endcsname\relax
  \providecommand{\doi}[1]{DOI~\discretionary{}{}{}#1}\else
  \providecommand{\doi}{DOI~\discretionary{}{}{}\begingroup
  \urlstyle{rm}\Url}\fi

\bibitem{BB09}
Bezhanishvili, G., Bezhanishvili, N.: An algebraic approach to canonical
  formulas: intuitionistic case.
\newblock Rev. Symb. Log. \textbf{2}, 517--549 (2009).


\bibitem{BB11}
Bezhanishvili, G., Bezhanishvili, N.: An algebraic approach to canonical
  formulas: modal case.
\newblock Studia Logica \textbf{99}, 93--125 (2011).


\bibitem{BB12}
Bezhanishvili, G., Bezhanishvili, N.: Canonical formulas for w{K}4.
\newblock Rev. Symb. Log. \textbf{5}, 731--762 (2012)

\bibitem{Bezhanishvili2014a}
Bezhanishvili, G., Bezhanishvili, N.: Locally finite reducts of {H}eyting
  algebras and canonical formulas.
\newblock Notre Dame J. Form. Log.  (in press).

\bibitem{BBI14}
Bezhanishvili, G., Bezhanishvili, N., Iemhoff, R.: Stable canonical rules.
\newblock J. Symb. Log.  (in press).

\bibitem{Nic06}
Bezhanishvili, N.: Lattices of intermediate and cylindric modal logics.
\newblock Ph.D. thesis, University of Amsterdam (2006)

\bibitem{Nic08}
Bezhanishvili, N.: Frame based formulas for intermediate logics.
\newblock Studia Logica \textbf{90}, 139--159 (2008)

\bibitem{BGGJ15}
Bezhanishvili, N., Gabelaia, D., Ghilardi, S., Jibladze, M.: Admissible bases
  via stable canonical rules.
\newblock Studia Logica \textbf{104}, 317--341 (2016)

\bibitem{BG14}
Bezhanishvili, N., Ghilardi, S.: Multiple-conclusion rules, hypersequents
  syntax and step frames.
\newblock In: R.~Gore, B.~Kooi, A.~Kurucz (eds.) Advances in Modal Logic (AiML
  2014), pp. 54--61. College Publications (2014)

\bibitem{BdJ16}
Bezhanishvili, N., de~Jongh, D.: Stable formulas in intuitionistic logic.
\newblock Notre Dame J. Form. Log.  (to appear)

\bibitem{blok1989algebraizable}
Blok, W.J., Pigozzi, D.: Algebraizable logics.
\newblock  Mem. Amer. Math. Soc. \textbf{77} (1989)

\bibitem{blok2002finite}
Blok, W.J., Van~Alten, C.J.: The finite embeddability property for residuated
  lattices, pocrims and {BCK}-algebras.
\newblock Algebra Universalis \textbf{48}, 253--271 (2002)

\bibitem{MR648287}
Burris, S., Sankappanavar, H.P.: A Course in Universal Algebra.
\newblock Springer-Verlag, New York-Berlin (1981)

\bibitem{CZ97}
Chagrov, A., Zakharyaschev, M.: Modal Logic.
\newblock The Clarendon Press, New York (1997)

\bibitem{APT:LICS}
Ciabattoni, A., Galatos, N., Terui, K.: From axioms to analytic rules in
  nonclassical logics.
\newblock Proceedings of LICS'08 pp. 229--240 (2008)

\bibitem{APT:AU}
Ciabattoni, A., Galatos, N., Terui, K.: Macneille completions of fl-algebras.
\newblock Algebra Universalis \textbf{66}, 405--420 (2011).


\bibitem{APT:APAL}
Ciabattoni, A., Galatos, N., Terui, K.: Algebraic proof theory for
  substructural logics: cut-elimination and completions.
\newblock Ann. Pure Appl. Logic \textbf{163}, 266--290 (2012).

\bibitem{Citkin14}
Citkin, A.: Characteristic formulas 50 years later (an algebraic account).
\newblock Arxiv  (2014).

\bibitem{GJ13}
Galatos, N., Jipsen, P.: Residuated frames with applications to decidability.
\newblock  Trans. Amer. Math. Soc. \textbf{365},
  1219--1249 (2013)

\bibitem{GJKO}
Galatos, N., Jipsen, P., Kowalski, T., Ono, H.: Residuated Lattices: an
  Algebraic Glimpse at Substructural Logics.
\newblock Elsevier (2007)

\bibitem{Jan68}
Jankov, V.: The construction of a sequence of strongly independent
  superintuitionistic propositional calculi.
\newblock Soviet Math. Dokl. \textbf{9}, 806--807 (1968)

\bibitem{Jer09}
Je{\v{r}}{\'a}bek, E.: Canonical rules.
\newblock J. Symb. Log. \textbf{74}, 1171--1205 (2009) 

\bibitem{Jerabek:15}
Je\v{r}\'{a}bek, E.: A note on the substructural hierarchy.
\newblock Preprint available at http://users.math.cas.cz/~jerabek/papers/shierarchy.pdf

\bibitem{kowalski2000splittings}
Kowalski, T., Ono, H.: Splittings in the variety of residuated lattices.
\newblock Algebra Universalis \textbf{44}, 283--298 (2000)

\end{thebibliography}
\end{document}